\documentclass[a4paper]{article}
\usepackage[utf8]{inputenc}
\usepackage{amsthm}
\usepackage{amssymb}
\usepackage{amsmath}
\usepackage{epsfig}
\usepackage{colonequals}
\usepackage{enumerate}
\usepackage{hyperref}
\usepackage{subcaption}
\usepackage{float}

\usepackage{amsfonts}       
\usepackage{bbding}         
\usepackage{bm}             
\usepackage{graphicx}       
\usepackage{fancyvrb}       
\usepackage{dcolumn}        
\usepackage{booktabs}       
\usepackage{paralist}       
\usepackage[usenames]{xcolor}  
\usepackage{amssymb} 
\usepackage{algpseudocode} 

\floatstyle{boxed}
\newfloat{Algorithm}{t}{lop}

\theoremstyle{plain}
\newtheorem{theorem}{Theorem}
\newtheorem{lemma}[theorem]{Lemma}

\newtheorem{corollary}[theorem]{Corollary}
\newtheorem{observation}[theorem]{Observation}

\theoremstyle{plain}

\theoremstyle{remark}

\algnewcommand{\LIf}[1]{\State\algorithmicif\ #1\ \algorithmicthen}
\algnewcommand{\EndLIf}{\unskip\ \algorithmicend\ \algorithmicif}
\algnewcommand{\LForAll}[1]{\State\algorithmicfor\ #1\ \algorithmicdo}
\algnewcommand{\EndLFor}{\unskip\ \algorithmicend\ \algorithmicfor}

\begin{document}
\title{Irreducible $4$-critical triangle-free toroidal graphs\footnote{Supported by the Neuron Foundation for Support of Science under Neuron Impuls programme.  An extended abstract of this paper appeared in proceedings of Eurocomb'17.}}
\author{%
     Zdeněk Dvořák\thanks{Computer Science Institute (CSI) of Charles University,
           Malostranské náměstí 25, 118 00 Prague, 
           Czech Republic. E-mail: \protect\href{mailto:rakdver@iuuk.mff.cuni.cz}{\protect\nolinkurl{rakdver@iuuk.mff.cuni.cz}}.
	   }
    \and
    Jakub Pekárek\thanks{Charles University,
           Malostranské náměstí 25, 118 00 Prague, 
           Czech Republic. E-mail: \protect\href{mailto:edalegos@gmail.com}{\protect\nolinkurl{edalegos@gmail.com}}.}
}
\date{\today}
\maketitle
\begin{abstract}
The theory of Dvořák, Král', and Thomas~\cite{trfree4} shows that a $4$-critical triangle-free graph
embedded in the torus has only a bounded number of faces of length greater than $4$ and that the
size of these faces is also bounded.  We study the natural reduction in such embedded graphs---identification
of opposite vertices in $4$-faces.  We give a computer-assisted argument showing that
there are exactly four $4$-critical triangle-free \emph{irreducible} toroidal graphs in which
this reduction cannot be applied without creating a triangle.  Using this result, we show
that every $4$-critical triangle-free graph embedded in the torus has at most four $5$-faces, or
a $6$-face and two $5$-faces, or a $7$-face and a $5$-face, in addition to at least seven $4$-faces.
This result serves as a basis for the exact description of $4$-critical triangle-free toroidal graphs,
which we present in a followup paper.
\end{abstract}

\section{Introduction}

The subject of coloring graphs on surfaces goes back to the work of Heawood~\cite{Heawood}, who proved that any graph $G$ drawn in surface $\Sigma$ of Euler genus $g>0$ is
$$H(\Sigma) \colonequals \lfloor (7 + \sqrt{24g + 1})/2\rfloor.$$
Franklin~\cite{franklin} and Ringel and Youngs~\cite{ringel} later showed that the bound is best possible for all surfaces except the Klein bottle,
for which the correct bound is 6. Incidentally, the assertion holds for the sphere as well, as stated by the Four-Color Theorem~\cite{AppHak1,AppHakKoc,rsst}. 

While Heawood's formula gives a tight bound on the possible values of chromatic
number of graphs on almost all surfaces, values close to the bound are achieved
by only relatively few graphs. An improvement of Heawoods's formula in this
sense was brought by Dirac~\cite{diracmaps} and Albertson and Hutchinson~\cite{Albhut} who showed that the
graph drawn in $\Sigma$ has chromatic number exactly $H(\Sigma)$ if and only if it contains a subgraph
isomorphic to the complete graph on $H(\Sigma)$ vertices. 

Further improvements are possible for large enough graphs. A graph $G$ is
\emph{k-critical} if its chromatic number is exactly $k$ and every proper
subgraph of $G$ has chromatic number at most $k-1$.
The importance of this notion comes from the fact that a graph is $k$-colorable
if and only if it does not contain a $k$-critical subgraph.
It follows from Euler's formula that if $\Sigma$ is a fixed surface and a graph $G$ drawn in $\Sigma$
has sufficiently many vertices, then $G$ has a vertex of degree at most six.
Consequently, for every $k \geq 8$, all $k$-critical graphs drawn in $\Sigma$
have a bounded number of vertices, and thus there are only finitely many such $k$-critical graphs.
A similar argument shows that this also is the case for $k = 7$. Hence, large chromatic number
of a graph embedded in a surface is always forced by one of finitely many obstructions.

A much more involved argument of Thomassen~\cite{Thomassen97} shows that
for every surface there are also only finitely many 6-critical graphs that can be
drawn in $\Sigma$. An immediate
consequence from a computational point of view is that for every $k \geq 6$,
fixed surface $\Sigma$ and a graph $G$ drawn in $\Sigma$ it is possible to
efficiently test whether $G$ is $(k-1)$-colorable in linear time, by testing the presence of
all possible $k$-critical subgraphs (the linear-time complexity can be achieved by using the
subgraph testing algorithm of Eppstein~\cite{eppstein00}). Such algorithm can
be constructed if an explicit full list of $k$-critical graphs on $\Sigma$ is
provided. The lists of 6-critical graphs are explicitly known for the
projective plane~\cite{albertson1979three}, the torus~\cite{Tho5torus} and
the Klein bottle~\cite{KawKraKynLid,ChePosStrThoYer}. 

Since the problem of testing 2-colorability is polynomial-time solvable, and
the problem of testing 3-colorability for planar graphs is NP-complete~\cite{garey1979computers},
the only remaining non-trivial case is
4-colorability. It is a long-standing open problem whether there is a polynomial time
algorithm for testing 4-colorability of graphs in a fixed surface $\Sigma$
other than the sphere. However a characterization by finitely many obstructions similar to the one described
for $(\geq5)$-colorability above does not exist, as shown by an elegant
construction of Fisk~\cite{Fisk78}. 

Let us consider the analogous problem for embedded graphs of larger girth.
Chromatic number of graphs of girth at least five is characterized by a deep
theorem of Thomassen~\cite{thomassen-surf} who showed that for every $k \geq 4$ and every
surface $\Sigma$ there are only finitely many $k$-critical graphs of girth at
least five that can be drawn in $\Sigma$. Thus testing $(k-1)$-colorability of
graphs of girth at least five again reduces to deciding the presence of
finitely many obstructions for any $k \geq 4$. There turn out to be no
4-critical graphs of girth at least five in the projective plane and the torus~\cite{thom-torus}
and in the Klein bottle~\cite{tw-klein} (i.e., all graphs of girth five drawn in one of these
surfaces are $3$-colorable).  For torus, Thomassen~\cite{thom-torus}
actually proved the following stronger claim.

\begin{theorem}[\cite{thom-torus}]\label{Thomassen}
If $G$ is a graph drawn in torus such that all contractible cycles have length at least five, then $G$ is 3-colorable. 
\end{theorem}

We now turn our attention to the main topic of this paper, graphs of girth at least $4$, i.e., triangle-free graphs.
It is easy to see that for $k\ge 5$, only finitely many $k$-critical triangle-free graphs
can be drawn in any fixed surface.  
Well-known Grötzsch' theorem~\cite{grotzsch1959} shows that every triangle-free planar graph is 3-colorable,
and thus there are no planar triangle-free $4$-critical graphs.
However, for other surfaces the case $k=4$ is much more involved.
For instance, the graphs obtained from an odd cycle of length
five or more by applying Mycielski's construction provide an
infinite class of 4-critical triangle-free graphs embeddable in any surface other than the
sphere. This of course means that 3-colorability of triangle-free graphs on a
fixed surface cannot be decided by testing the presence of finitely
many obstructions.

The only non-planar surface for which the 3-colorability problem for
triangle-free graphs is fully characterized is the projective plane. Building
on earlier work of Youngs~\cite{Youngs}, Gimbel and Thomassen~\cite{gimbel} obtained an elegant
characterization stating that a triangle-free graph drawn in the projective
plane is 3-colorable if and only if it has no subgraph isomorphic to a
non-bipartite quadrangulation of the projective plane.

For other surfaces, only the following approximate characterizations are known.
For a graph $G$ embedded in a surface, let $S(G)$ denote the multiset of
lengths of $(\geq\!5)$-faces of $G$ (thus, the characterization from the previous paragraph
implies that $S(G)=\emptyset$ for every $4$-critical
projective-planar triangle-free graph $G$). Dvořák, Král', and Thomas~\cite{trfree4} proved that for any
surface $\Sigma$, there exists a constant $c_\Sigma$ such that every
$4$-critical triangle-free graph $G$ embedded in $\Sigma$ \emph{without
non-contractible 4-cycles} satisfies $\sum S(G)\le c_\Sigma$; i.e., $G$ has
only a bounded number of faces of length greater than $4$ and these faces have
bounded lengths. Such a bound does not hold in general if non-contractible
$4$-cycles are allowed (but it does hold for toroidal graphs, as we will see
below). A more detailed treatment of $4$-critical triangle-free graphs with
non-contractible $4$-cycles was given by Dvořák and Lidický~\cite{cylgen-part3}.
Dvořák, Král', and Thomas~\cite{trfree6}
proved that for any surface $\Sigma$, a triangle-free graph embedded in
$\Sigma$ with large edgewidth is $3$-colorable unless $\Sigma$ is
non-orientable and the graph contains a quadrangulation with an odd orienting
cycle. They also designed a linear-time algorithm to test $3$-colorability of
embedded triangle-free graphs~\cite{trfree7}.

\begin{figure}[!ht]
\centering

\begin{subfigure}{0.4\textwidth}
\includegraphics[width=140pt]{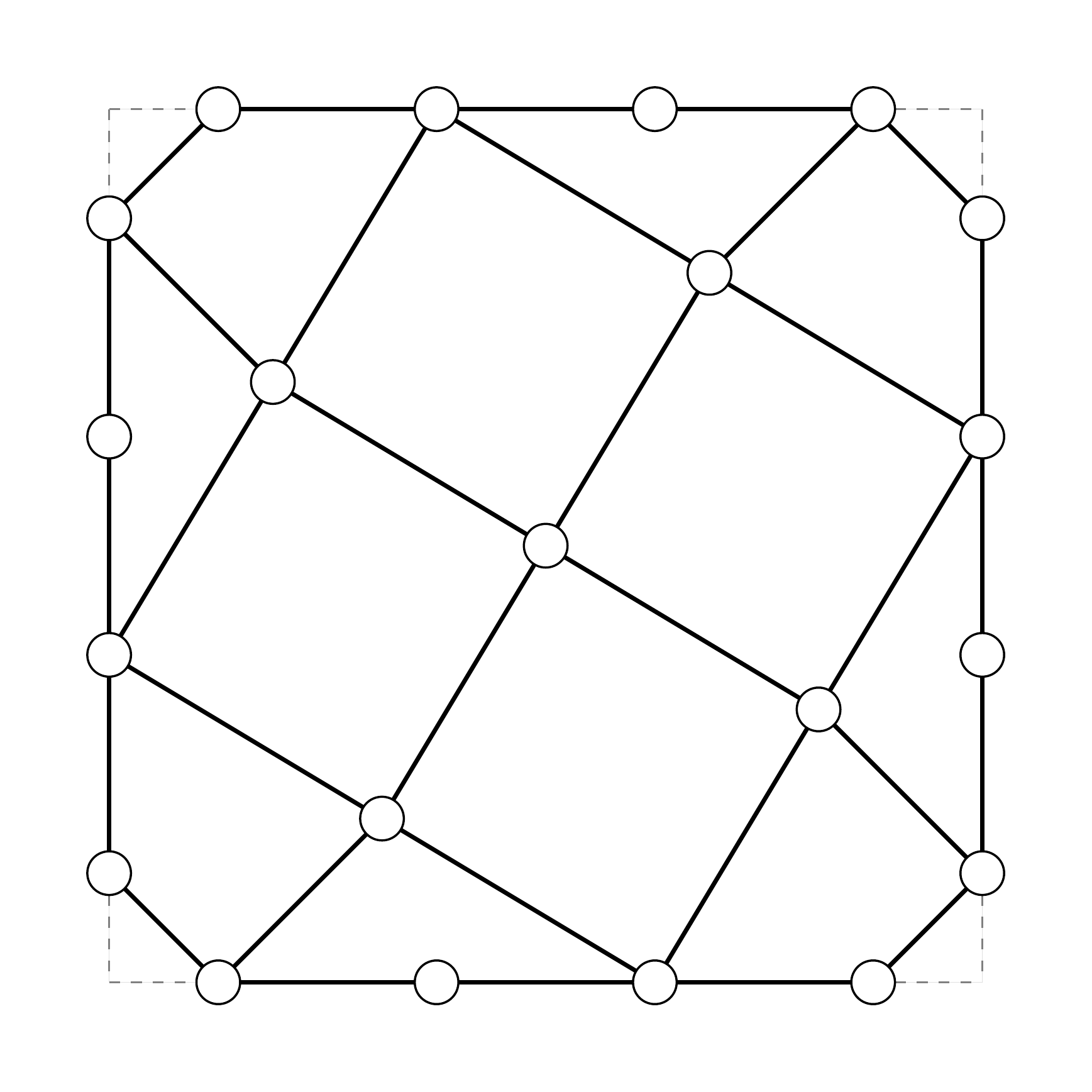}
\caption{Graph $I_4$}
\end{subfigure}
\begin{subfigure}{0.4\textwidth}
\includegraphics[width=140pt]{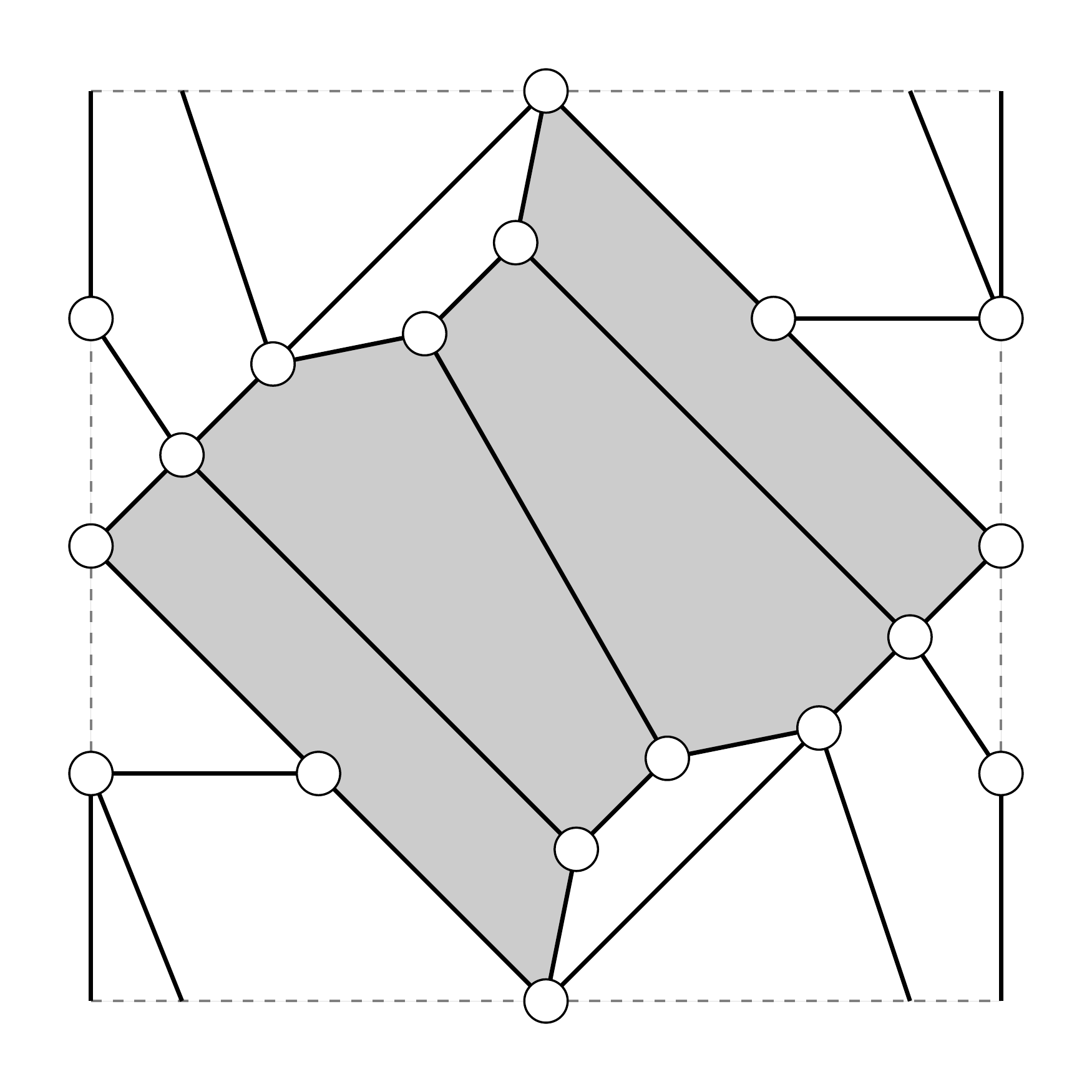}
\caption{Graph $I_5$}
\end{subfigure}

\begin{subfigure}{0.4\textwidth}
\includegraphics[width=140pt]{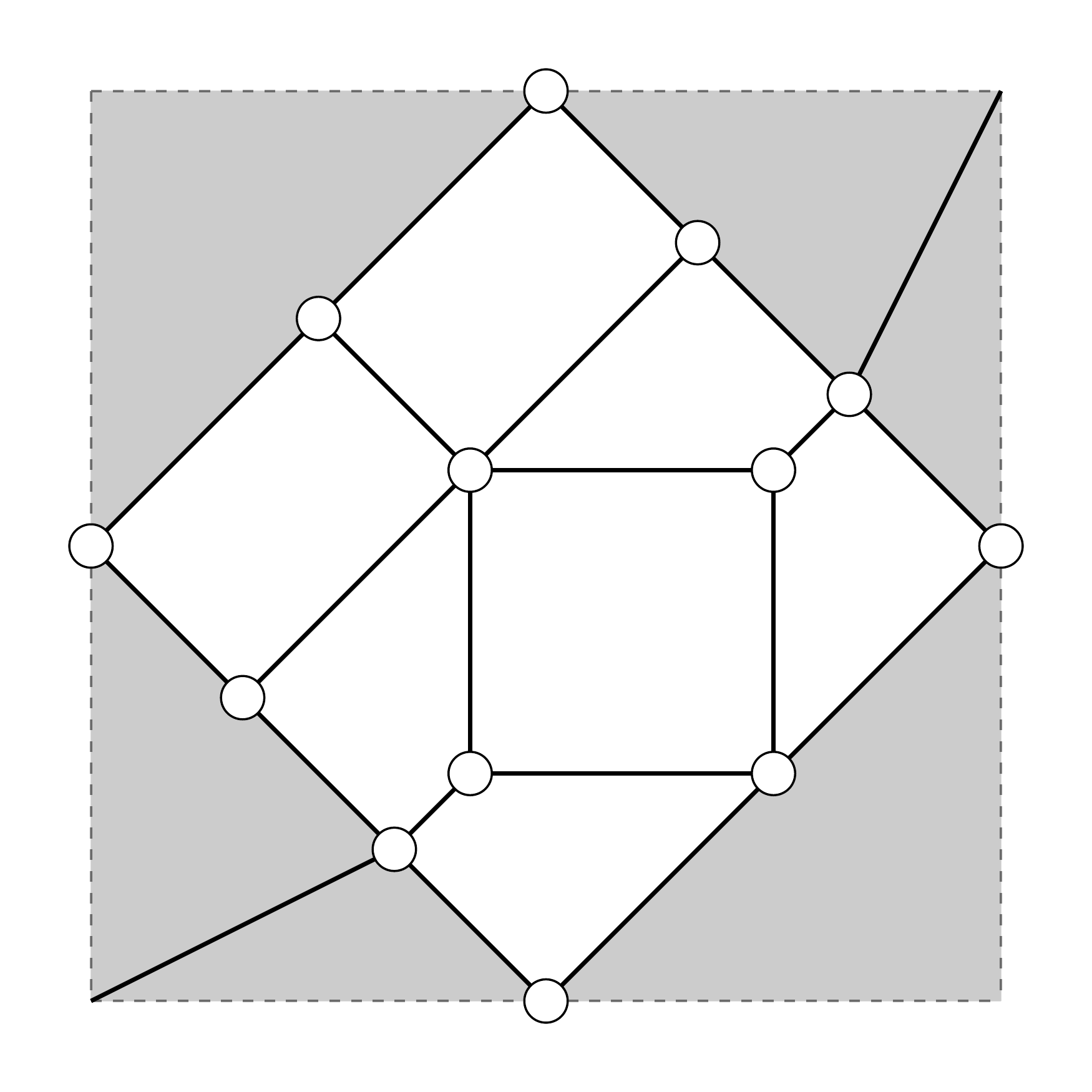}
\caption{Graph $I_7^{a}$}
\end{subfigure}
\begin{subfigure}{0.4\textwidth}
\includegraphics[width=140pt]{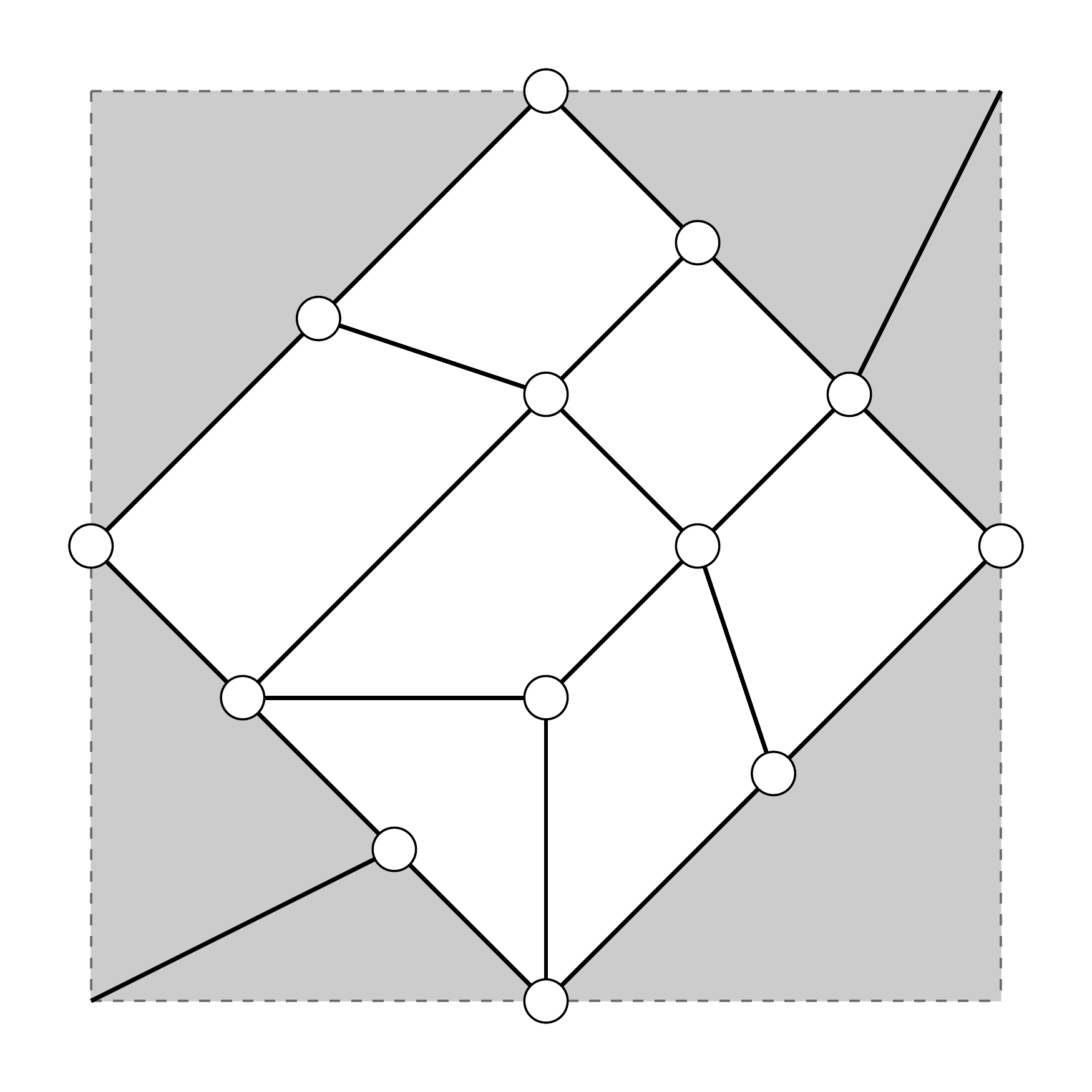}
\caption{Graph $I_7^{b}$}
\end{subfigure}

\caption{Irreducible 4-critical graphs drawn in torus; shadowed faces have length greater than $4$.}
\label{figIr}
\end{figure}

In this paper, we focus on the toroidal case.
Král' and Thomas~\cite{thomas2008coloring} proved that there is only one $4$-critical
triangle-free graph embedded in the torus with all faces of even length
(depicted as $I_4$ in Figure~\ref{figIr}). On
the other hand, the theory of~\cite{trfree4} can be used to show that if $G$ is
a $4$-critical triangle-free graph drawn in torus,
then $\sum S(G)\le 500$.  Our main result is a substantial strengthening of this bound.

\begin{theorem}\label{thm-main}
If $G$ is a 4-critical triangle-free graph drawn in torus, then
$S(G) = \{7,5\}$, or
$S(G) = \{6,5,5\}$, or
$S(G) = \{5,5,5,5\}$, or
$S(G) = \{5,5\}$, or
$S(G) = \emptyset$ and $G$ is the graph $I_4$ depicted in Figure~\ref{figIr}.
Furthermore, $G$ has at least seven $4$-faces and representativity at least $2$.
\end{theorem}
Let us remark that all the multisets of face lengths mentioned in Theorem~\ref{thm-main} are realized by some 4-critical triangle-free toroidal
graph (in fact, infinitely many except for the last case).

\subsection{Reductions of $4$-faces and irreducible graphs}

The results surveyed above make it clear that $4$-faces play an important role in $4$-critical triangle-free graphs.
In particular, Theorem~\ref{Thomassen} implies that every $4$-critical
triangle-free graph drawn in the torus has a $4$-face;
and actually, we can show that the $4$-faces cover all vertices of the graph.
\begin{lemma}\label{4FaceCover}
If $G$ is a 4-critical graph drawn in torus such that every triangle is non-contractible, then every vertex of $G$ is incident with a 4-face.  
\end{lemma}

There is a natural reduction operation on $4$-faces: identifying a pair of opposite vertices on a $4$-face to a single vertex
cannot decrease the chromatic number of the graph.  We say that an embedded graph is \emph{irreducible}
if each such identification creates a triangle.  Theorem~\ref{thm-main} is proved by studying the inverse
process to this reduction, using the irreducible 4-critical triangle-free toroidal graphs as the basic case for an inductive argument.
To carry out this idea, we need an explicit list of such irreducible graphs, which we obtain via computer-assisted enumeration.

\begin{theorem}\label{IrEnum4}
Each irreducible 4-critical triangle-free toroidal graph is isomorphic to one of the graphs depicted in Figure~\ref{figIr}.
\end{theorem}

In a followup paper~\cite{dpfuture}, we build upon these results to obtain (again using computer-assisted enumeration) a full characterization of 4-critical toroidal triangle-free graphs.

\section{Preliminaries}

The graphs we consider have no parallel edges
and no loops, while multigraphs are allowed to have both parallel edges and loops.
We always consider graphs associated with a specific drawing on torus. We
explicitly acknowledge this by stating that $G$ is \emph{drawn in torus}. The
drawing $\delta$ of a graph $G$ in a surface $\Sigma$ is defined as a pair of functions
$\delta : V(G) \rightarrow \Sigma$ associating distinct vertices of $G$ with
distinct points in $\Sigma$, and $\delta : E(G) \rightarrow 2^\Sigma$ associating distinct
edges of $G$ with pairwise disjoint open arcs in $\Sigma\setminus \delta(V(G))$ such that for every $uv
\in E(G)$ we have $\overline{\delta(uv)} = \delta(uv) \cup \delta(u) \cup \delta(v)$ (where bar denotes the closure).
We consider two graphs drawn in a surface isomorphic if and only if their
drawings can be transformed into one another by a homeomorphism of the surface.

The arcwise connected components of the surface minus the drawing are called
\emph{faces}. A face is \emph{2-cell}, if it is homeomorphic to an open disk.
The boundary of a 2-cell face is an image of a closed walk called the
\emph{facial walk}.  A closed walk is \emph{facial} if it is a facial walk of
some face.  A 2-cell face is a \emph{k-face} if its facial walk has
length $k$. A drawing is a \emph{2-cell} drawing if all of its faces are 2-cell
and thus have a single facial walk. Generally we will only consider 2-cell
drawings of graphs. 

Let $G$ be a graph and let $C$ be a proper subgraph of $G$. For a positive integer $k$,
we say that $G$ is \emph{$C$-critical for $k$-coloring}
if for every proper subgraph $H \subsetneq G$ such that $C \subseteq H$,
there exists a $k$-coloring of $C$ that extends to a $k$-coloring of $H$,
but not to a $k$-coloring of $G$.  Throughout the rest of the paper, we only consider $3$-colorings,
and thus we will omit the ``for $3$-coloring'' part for brevity.  We need the following basic
fact about critical graphs.

\begin{lemma}\label{SubgrCrit}
Let $G$ be a graph drawn in surface $\Sigma$, let $\Lambda\subseteq\Sigma$ be an open
disk whose boundary traces a closed walk in $G$, let $C$ be the subgraph of $G$ formed by
vertices and edges contained in the boundary of $\Lambda$, and let $G_1$ be the subgraph of $G$ drawn in the closure of
$\Lambda$. If $G$ is 4-critical and $\Lambda$ is not its face, then $G_1$ is
$C$-critical.
\end{lemma}
\begin{proof}
Since $\Lambda$ is not a face of $G$, $C$ is a proper subgraph of $G_1$.
Let $G_2$ denote the subgraph of $G$ drawn in the closure of the complement of
$\Lambda$.
Consider any proper subgraph $H\supseteq C$ of $G_1$.  Since $G$ is $4$-critical,
its proper subgraph $H\cup G_2$ has a $3$-coloring $\varphi$; and furthermore,
$\varphi$ does not extend to a $3$-coloring of $G$.  Consequently, the
restriction of $\varphi$ to $C$ does not extend to a $3$-coloring of $G_1$.
Since this holds for every choice of $H$, we conclude that $G_1$ is $C$-critical.
\end{proof}

Gimbel and Thomassen~\cite{gimbel} characterized
plane triangle-free graphs whose outer face is bounded by a $6$-cycle $C$ and some $3$-coloring
of $C$ does not extend to a $3$-coloring of the whole graph.  In terms of $C$-critical graphs,
their result can be stated as follows.

\begin{theorem}[\cite{gimbel}]\label{Outer}
Let $G$ be a triangle-free plane graph with the outer face bounded by
a cycle $C$ of length at most $6$. Then $G$ is $C$-critical if and only if $C$ is
a $6$-cycle, all internal faces of $G$ have length exactly four and $G$ contains
no separating 4-cycle.
\end{theorem}

Together with Lemma~\ref{SubgrCrit}, this has the following consequence.

\begin{corollary}\label{Separations}
Let $G$ be a 4-critical graph drawn in a surface $\Sigma$ with no contractible triangles. 
Let $\Lambda$ be a subset of $\Sigma$ homeomorphic to an open disk, such that the boundary of $\Lambda$ traces a closed walk $W$ in $G$.
If $|W|\le 5$, then $\Lambda$ is a face of $G$.  If $|W|=6$, then either $\Lambda$ is a face of $G$
or all faces contained in $\Lambda$ are 4-faces. 
\end{corollary}

\section{$4$-faces in critical graphs}\label{SecDef}

Let $G$ be a graph drawn in a surface $\Sigma$ with a 4-face $f$ bounded by a cycle $v_1v_2v_3v_4$.
An \emph{$f$-deflation of $G$ (of form $(v_1v_2v_3v_4) \rightarrow (v_2v'v_4)$)}
is the multigraph $H$ obtained from $G$ by identifying $v_1$ with $v_3$ to a new vertex $v'$
within the face $f$ and by suppressing the resulting $2$-faces $v_2v'$ and $v_4v'$.
We say that $H$ is a \emph{deflation} of $G$ if it is an $f$-deflation for some $4$-face $f$ of $G$.
Every $3$-coloring of $H$ corresponds to a $3$-coloring of $G$ obtained by giving $v_1$ and $v_3$ the color of $v'$;
hence, if $G$ is not $3$-colorable, then $H$ is not $3$-colorable either.

Let us remark that the deflation may have parallel edges in case $v_1$ and $v_3$ have a common neigbor
other than $v_2$ and $v_4$.  Nevertheless, since we consider $H$ being drawn in $\Sigma$, we can distinguish
edges that were incident with $v_1$ from edges that were incident with $v_3$ before deflation,
based on which side of the path $v_2v'v_4$ they attach to in $H$.  Consequently, the operation
of deflation is invertible (up to isomorphism of the drawings).

\begin{lemma}\label{DefNoTri}
Let $G$ be a 4-critical graph without contractible triangles drawn in a surface.
Then deflations of $G$ do not contain contractible triangles. Furthermore, if
$G$ is triangle-free, then deflations of $G$ do not contain loops.  
\end{lemma}
\begin{proof}
Let $f$ be a $4$-face of $G$ bounded by a cycle $v_1v_2v_3v_4$ and
let $H$ be an $f$-deflation of $G$ of form $(v_1v_2v_3v_4) \rightarrow (v_2v'v_4)$.
The deflation only creates a loop (at $v'$) if $v_1$ and $v_3$ are adjacent in $G$;
but in that case, $G$ contains a triangle $v_1v_2v_3$.  Hence, if $G$ is triangle-free,
then $H$ has no loops.

Suppose now that $H$ contains a contractible triangle $T$.  Since $G$ does not contain
contractible triangles, the triangle $T$ arose from a path $v_1xyv_3$ of length three in $G$.
Note that $x\neq v_2\neq y$, since $G$ is simple and does not contain contractible triangles.
Since $T$ is contractible, the $5$-cycle $C=v_1xyv_3v_2$ is contractible and bounds
an open disk $\Lambda\subseteq \Sigma$.  By Corollary~\ref{Separations},
$\Lambda$ is a face of $G$.  Consequently, the path $v_1v_2v_3$ is contained
in the boundaries of two faces $f$ and $\Lambda$ of $G$, and thus $v_2$ has degree $2$.
This contradicts the assumption that $G$ is $4$-critical.
\end{proof}

Let $G$ be a triangle-free $4$-critical graph drawn in a surface.
As we observed before, a deflation of $G$ is not $3$-colorable; however, it
is not necessarily $4$-critical.  A graph $F$ is a \emph{reduction} of $G$
if there exists a deflation $H$ of $G$ such that $H$ is triangle-free, $F$ is
a subgraph of $H$, and $F$ is $4$-critical.  By Lemma~\ref{DefNoTri},
$F$ has no loops, and since it is $4$-critical, it has no parallel edges;
hence, unlike deflations, reductions are simple graphs.

A triangle-free 4-critical graph $G$ drawn in a surface is \emph{reducible} if $G$
has a reduction; equivalently, $G$ is reducible if $G$ has a 4-face $f$ such that at least one of
the two possible $f$-deflations of $G$ is triangle-free.  Otherwise,
we say that $G$ is \emph{irreducible}.
We only define terms reducible and irreducible for the class of triangle-free
4-critical graphs. Therefore, when we state that a graph $G$ is (ir)reducible,
we implicitly state that it is also triangle-free and 4-critical. 

It is clear that the number of vertices and edges strictly decreases with every
iteration of the reduction operation. The same is true for the
number of 4-faces.  Let us show this in a slightly more general setting of graphs that may
have non-contractible triangles.
For a graph $G$ drawn in a surface, let $c(G)$ denote the number
of its 4-faces.

\begin{lemma}\label{4FaceDec}
Let $G$ be a 4-critical graph drawn in a surface $\Sigma$, with no contractible triangles.
Let $H$ be a $4$-critical subgraph of a deflation of $G$.  Then $c(H)<c(G)$.
Furthermore, if $C$ is a cycle bounding a $4$-face in $H$, then all vertices of $C$
are incident with $4$-faces in $G$.
\end{lemma}
\begin{proof}
Let $G'$ be an $f$-deflation of $G$ of form $(v_1v_2v_3v_4) \rightarrow (v_2v'v_4)$ such
that $H$ is a subgraph of $G'$.  Consider any $4$-face $g$ of $H$, bounded by a cycle $C$.
If $C$ is also a cycle in $G$ (up to renaming of $v'$ to $v_1$ or $v_3$), then $C$ also bounds
a face $g'$ in $G$, and all its vertices are incident with $g'$.

Otherwise, $C=v'xyz$ corresponds to a path $P=v_1xyzv_3$ in $G$ with $\{v_2,v_4\}\cap \{x,z\}=\emptyset$.
Since the edges $v'x$ and $v'z$ attach to the path $v_2v'v_4$ in $G'$ from opposite sides,
$g$ is not a face of $G'$ and by symmetry, we can assume that the edge $v_2v'$ is drawn inside the open disk $g$.
Consequently, the open disk $\Lambda\subseteq \Sigma$ bounded by the closed walk $v_1xyzv_3v_4$ in $G$ contains the edges
$v_2v_1$ and $v_2v_3$, and thus it does not bound a face.  By Corollary~\ref{Separations}, all faces of $G$
inside $\Lambda$ are $4$-faces, and thus all vertices of $C$ are incident with $4$-faces in $G$.
In this case, we let $g'$ be a $4$-face of $G$ contained in $\Lambda$ distinct from $f$.

The natural mapping $g\mapsto g'$ from $4$-faces of $H$ to $4$-faces of $G$ is injective and its image does not
contain $f$, and thus $c(H)<c(G)$.
\end{proof}

We are now ready to prove the claim that vertices of toroidal $4$-critical graphs with no contractible triangles
are covered by $4$-faces.

\begin{proof}[Proof of Lemma~\ref{4FaceCover}]
For contradiction, suppose that there exists a $4$-critical graph $G$ drawn in torus without contractible triangles
containing a vertex $v$ not incident with any $4$-face.   Let us choose such a graph with minimum number of
vertices.

According to Theorem~\ref{Thomassen} and Corollary~\ref{Separations}, $G$ has a 4-face, bounded by a cycle
$v_1v_2v_3v_4$.  Since $v$ is not incident with a $4$-face, we have $v\not\in \{v_1,\ldots, v_4\}$.
Since $G$ is 4-critical, $G-v$ has a 3-coloring $\psi$.
Without loss of generality, we can assume that $\psi(v_1)=\psi(v_3)$.  
Let $H$ denote a 4-critical subgraph of the deflation of $G$ of form
$(v_1v_2v_3v_4) \rightarrow (v_2v'v_4)$.  Since $\psi$ gives a proper 3-coloring
to $H-v$ and $H$ is not 3-colorable, it follows that $v \in V(H)$. 

According to Lemma~\ref{DefNoTri}, all triangles in $H$ are non-contractible.
By the minimality of $G$, we conclude that all vertices of $H$ are incident with $4$-faces,
and in particular, $v$ is incident with a $4$-face of $H$.  However,
then $v$ is incident with a $4$-face of $G$ by Lemma~\ref{4FaceDec}, which is a contradiction.
\end{proof}

\section{Enumeration of irreducible graphs}\label{SecEnu}

In this section, we describe a theoretical basis for enumeration of all
irreducible graphs drawn in torus via a computer program.
The \emph{base graph} $B$ is the graph obtained from $K_4$ by subdividing
edges of its perfect matching twice, with its unique embedding in the torus
that has a $4$-face. Let $f$ denote the 4-face and $C=v_1v_2v_3v_4$ its boundary cycle. We call
the two distinct paths of length 3 internally disjoint from $C$ \emph{links}. Each link
joins a pair of opposite vertices $v_i$ and $v_{i+2}$ in $C$, and together with a
path in $C$ connecting $v_i$ and $v_{i+2}$ forms a non-contractible 5-cycle. 

We say a 4-face $f$ of a graph $G$ drawn in torus is \emph{linked} if there exists a subgraph $H$
of $G$ such that $H$ is a base graph and $f$ is its 4-face. We say a graph $G$
is \emph{linked} if all of its 4-faces are linked. 
Furthermore, let $C = v_1v_2v_3v_4$ be a cycle bounding a 4-face. We say that a pair of
opposite vertices $v_1$ and $v_3$ is \emph{linked} if there exists a path $v_1xyv_3$ such
that $v_1xyv_3v_2$ is a non-contractible 5-cycle. 

\begin{lemma}\label{IrLinked}
Each irreducible graph $G$ drawn in torus is linked and 
contains the base graph as a subgraph.
\end{lemma}
\begin{proof}
Let $C=v_1v_2v_3v_4$ be a cycle bounding a $4$-face $f$ of $G$.  Since $G$ is irreducible, identifying $v_1$ with $v_3$
creates a triangle, which is by Lemma~\ref{DefNoTri} non-contracible; hence the pair $v_1$ and $v_3$ is linked, via some path $P_1$.
Similarly, $v_2$ and $v_4$ are linked via soma path $P_2$.  Since $G$ is triangle-free, the two paths are disjoint, and thus
$C\cup P_1\cup P_2$ is the base graph and $f$ is linked.  Since this holds for every $4$-face of $G$, we conclude that $G$ is linked.
Furthermore, by Theorem~\ref{Thomassen} and Corollary~\ref{Separations} $G$ has a $4$-face, and thus it contains the base graph as a subgraph.
\end{proof}

\begin{figure}

\begin{subfigure}{0.3\textwidth}
\centering
\includegraphics[width=120pt]{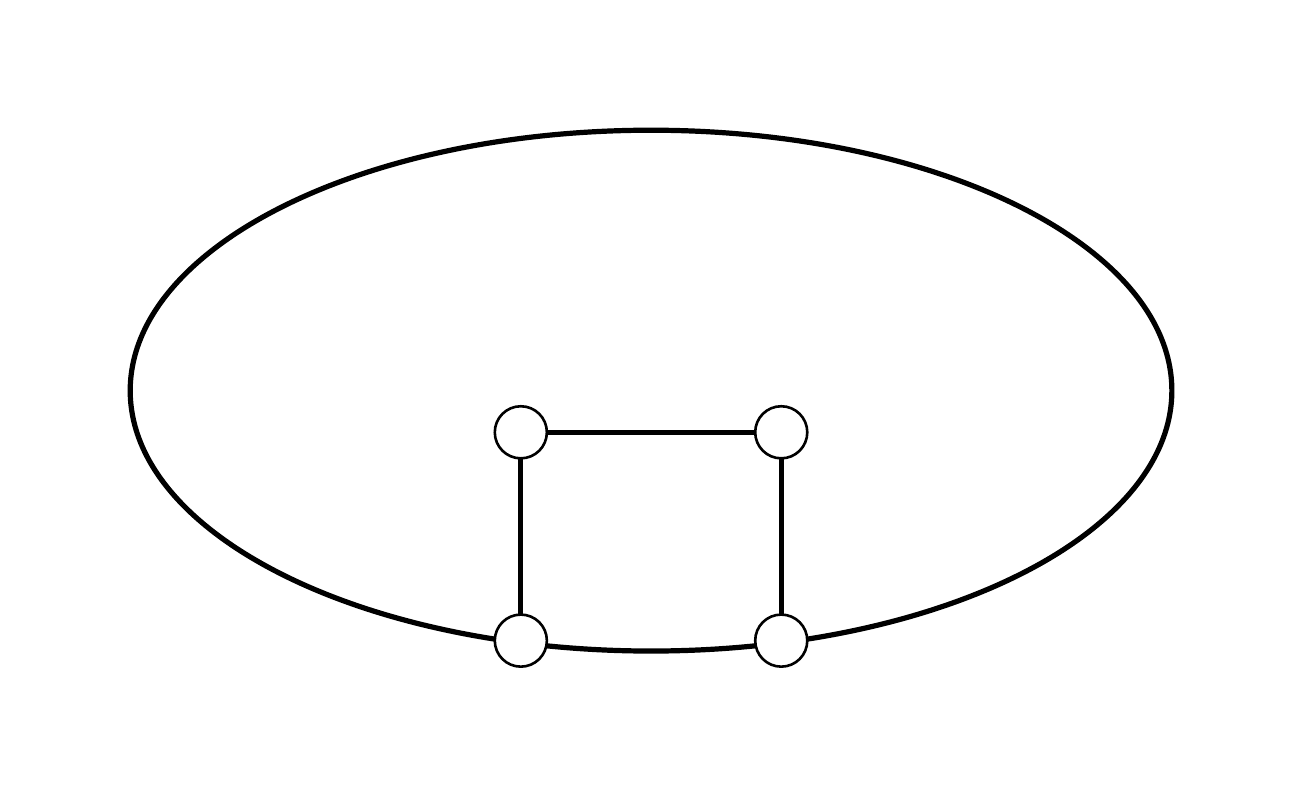}
\caption{One common edge}
\end{subfigure}
\begin{subfigure}{0.3\textwidth}
\centering
\includegraphics[width=120pt]{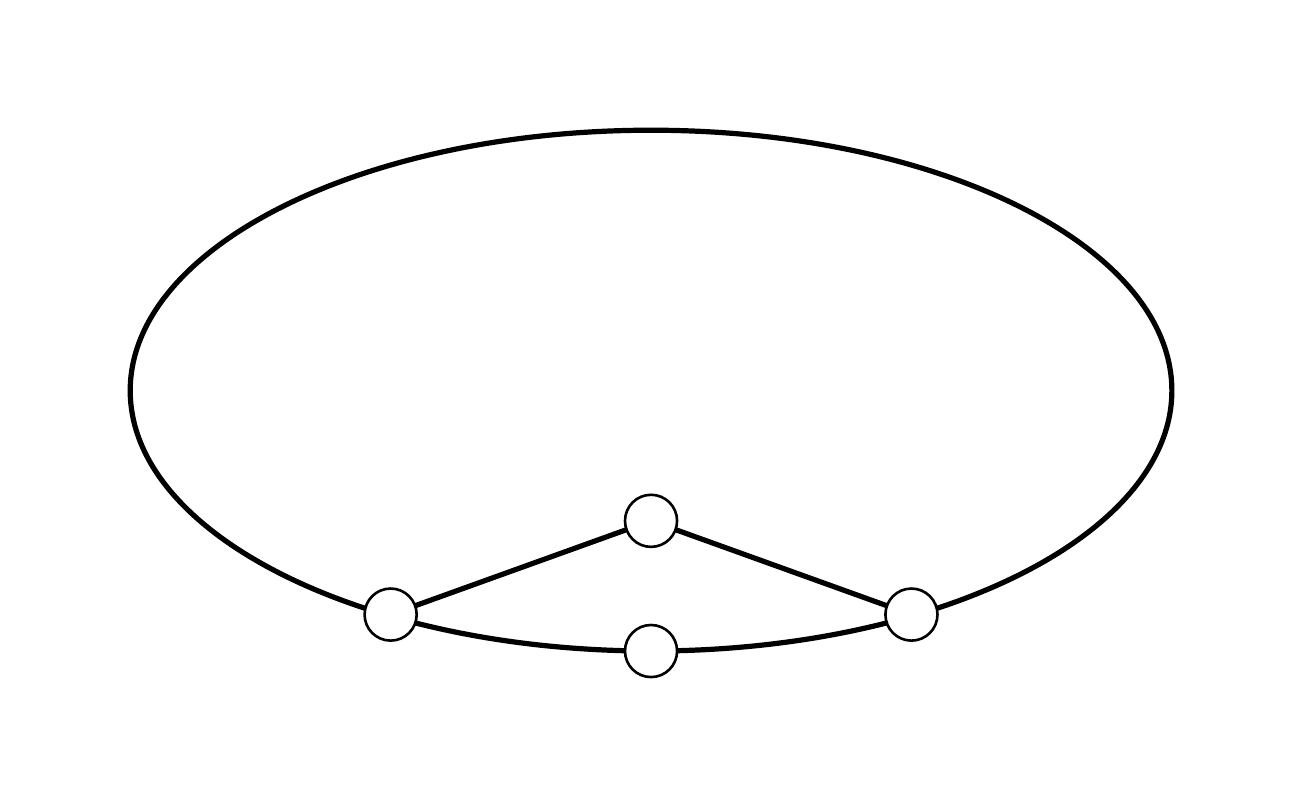}
\caption{Two common edges}
\end{subfigure}
\begin{subfigure}{0.3\textwidth}
\centering
\includegraphics[width=120pt]{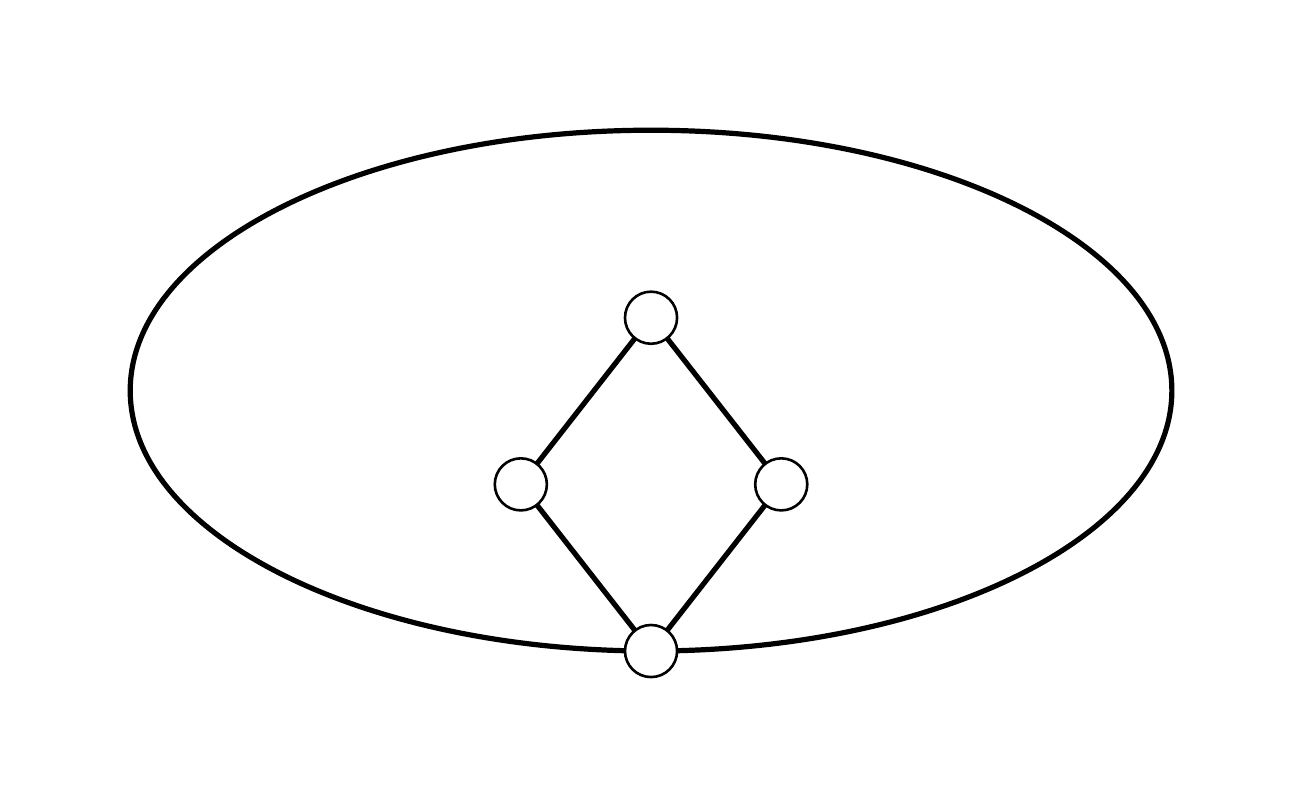}
\caption{One common vertex}
\end{subfigure}

\begin{subfigure}{0.3\textwidth}
\centering
\includegraphics[width=120pt]{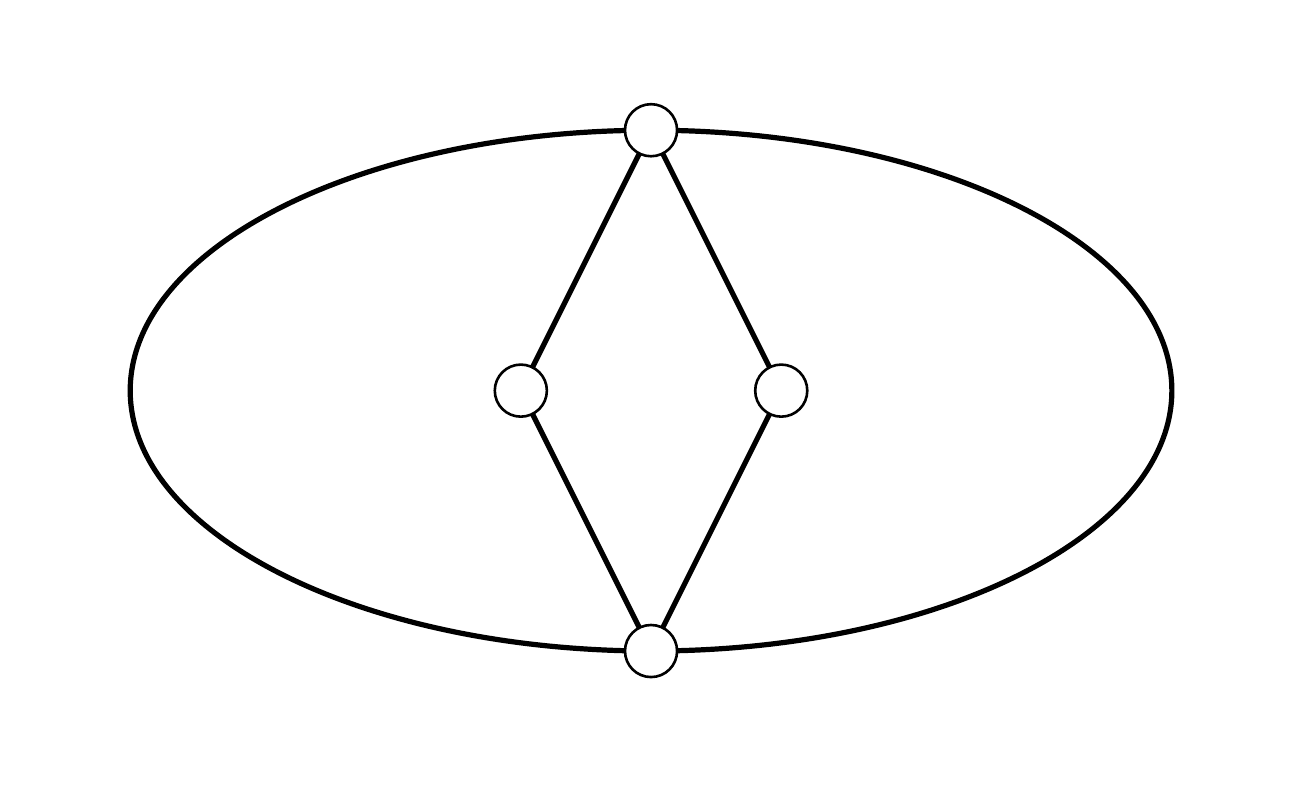}
\caption{Two common vertices}
\end{subfigure}
\begin{subfigure}{0.3\textwidth}
\centering
\includegraphics[width=120pt]{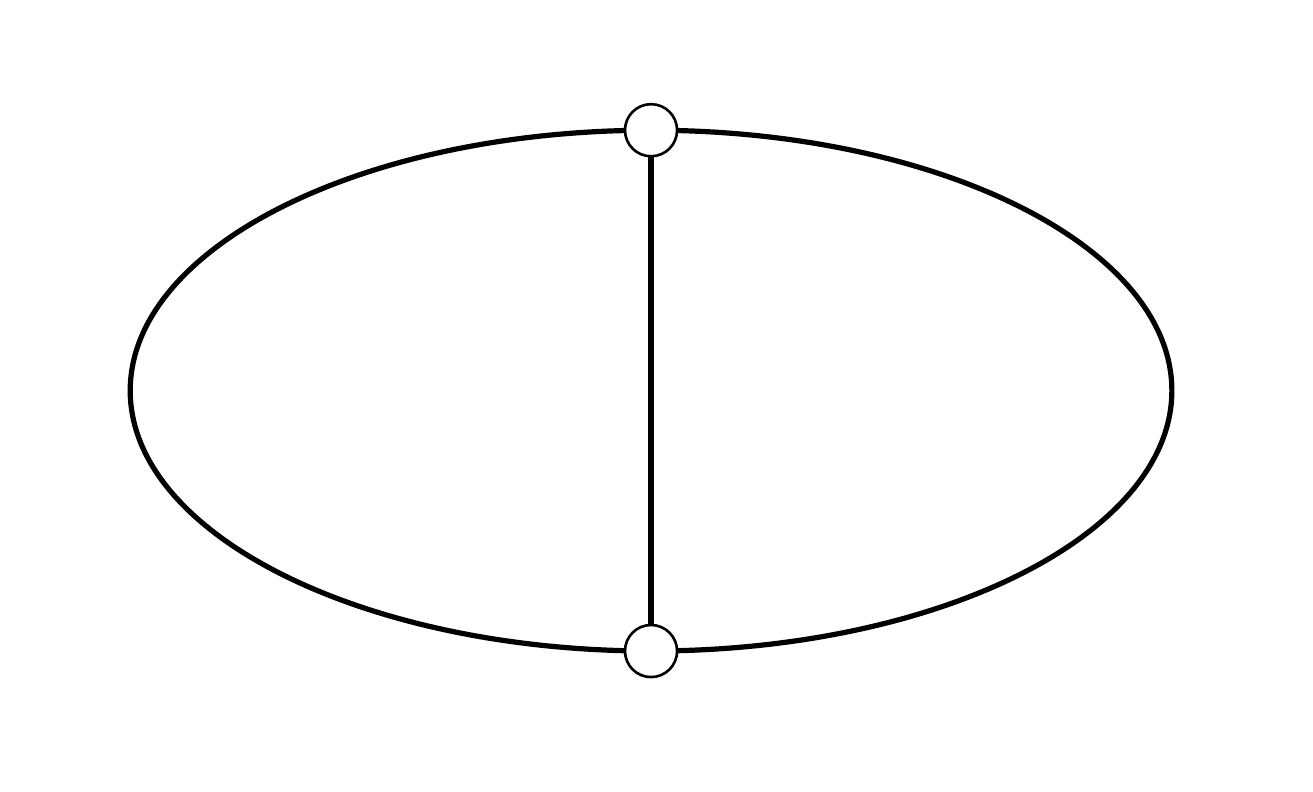}
\caption{Chord}
\end{subfigure}
\begin{subfigure}{0.3\textwidth}
\centering
\includegraphics[width=120pt]{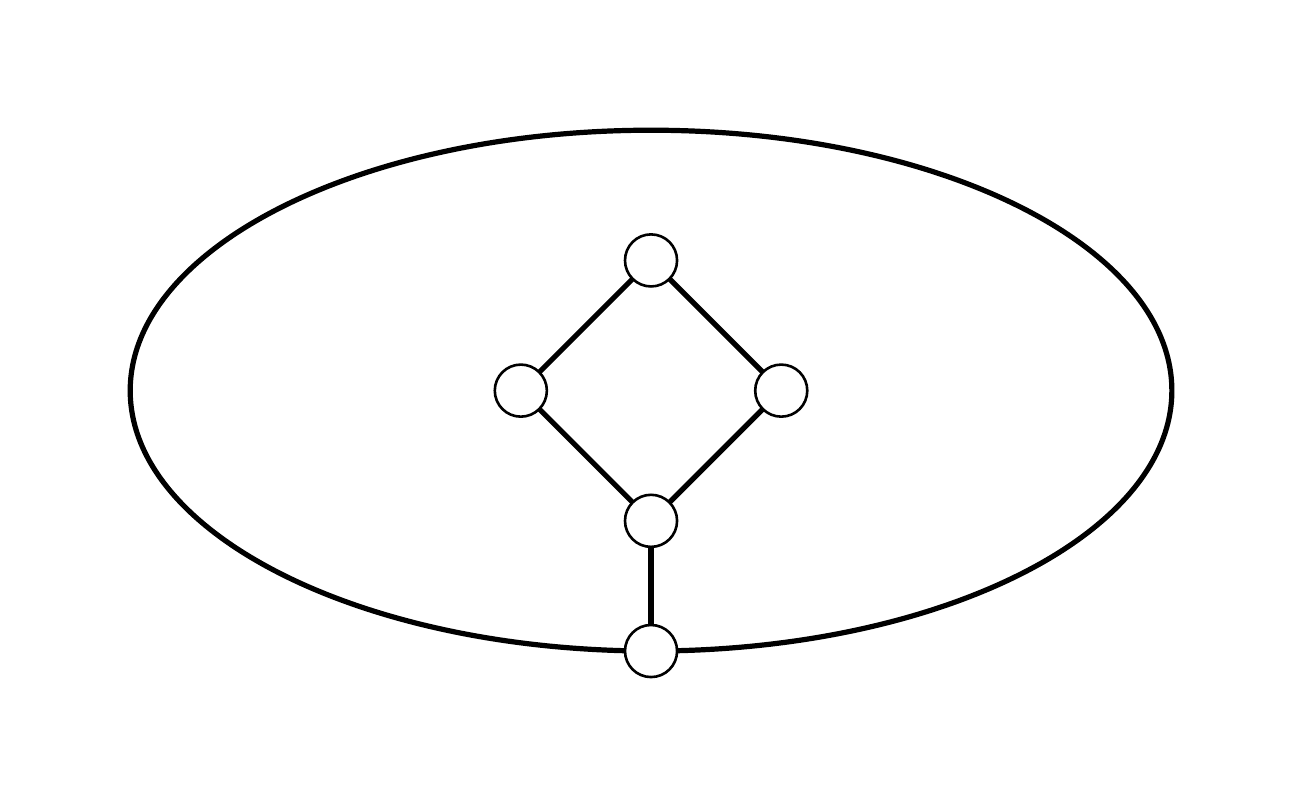}
\caption{No common elements}
\end{subfigure}
\caption{Expansion operations.} 
\label{figExp}
\end{figure}

Let $H$ be a graph with a $2$-cell drawing on torus. Let
$g$ be a face in $H$ of length at least six with facial walk $C$.
For a positive integer $k$, a \emph{$k$-chord} of $g$ is a path $P$ of length $k$ whose ends
are distinct vertices of $C$ and otherwise $P$ is drawn inside $g$.
Perform one of the following operations (see Figure~\ref{figExp} for illustration):
\begin{enumerate}
  \item[(a)] Add a $3$-chord of $g$ joining consecutive vertices of $C$.
  \item[(b)] For a subpath $v_1v_2v_3$ of $C$, add a $2$-chord of $g$ joining $v_1$ with $v_3$.
  \item[(c)] For a vertex $v$ in the boundary of $g$, add a walk $vxyzv$ drawn inside $g$.
  \item[(d)] For non-consecutive vertices $v$ and $z$ of the facial walk of $g$, add two internally disjoint $2$-chords joining $v$ with $z$.
  \item[(e)] Add a $1$-chord of $g$ joining non-consecutive vertices of $C$.
  \item[(f)] For a vertex $v$ in the boundary of $g$, add a $4$-cycle $wxyz$ drawn inside $G$ and an edge $vw$.
\end{enumerate}
If $H'$ is the result of applying one of the previous operations, then we say $H'$ is
an \emph{expansion} of $H$ (inside $g$).  Note that expansion of a graph is always its strict supergraph.
We now show that all irreducible graphs on torus can be obtained from the base graph by expansions.

\begin{lemma}\label{IrExpChar}
Let $G$ be an irreducible graph drawn in torus.  Let $H$ be a proper connected subgraph $H$ of $G$
containing the base graph as a subgraph.  For any face $g$ of $H$ which is not a face in $G$,
there exists an expansion $H'$ of $H$ in $g$ such that $H'\subseteq G$.
\end{lemma}
\begin{proof}
Since the drawing of the base graph is 2-cell, the drawing of its connected supergraph $H$ is
also $2$-cell.  Let $C$ be the boundary walk of $g$.  Corollary~\ref{Separations} implies that $|C|\ge 6$.
If some edge $e$ of $G$ with both ends in $V(C)$ is drawn inside $g$, then $H+e\subseteq G$
is an expansion of $H$ according to the construction (e).  Hence, suppose there is no such edge.

Consequently, there exists a vertex $v\in V(G)\setminus V(H)$ drawn inside $g$.
Since $G$ is $4$-critical, it is connected, and thus we can assume that $w$ has a neighbor $v$ in $C$.
By Lemma~\ref{4FaceCover}, $w$ is incident with a $4$-face in $G$; let $K$ be the cycle that bounds this $4$-face.
If $K\cap C=\emptyset$, then $H+K+vw\subseteq G$ is an expansion of $H$ according to the construction (f).

Hence, suppose $K\cap C\neq\emptyset$.  Since we assume that there is no edge inside $g$ joining two vertices
of $C$, we conclude that $K\cap C$ either is a path of length at most $2$ (possibly a single vertex) or consists of two non-adjacent vertices
of $K$.  In the former case, $H+K\subseteq G$ is an expansion of $H$ according to the constructions (a), (b), or (c).
In the latter case, $H+K\subseteq G$ is an expansion of $H$ according to the construction (d).
\end{proof}

As a consequence of Lemmas~\ref{IrLinked} and \ref{IrExpChar}, every irreducible graph on torus
can be created from the base graph by iterated expansions.  However, if we tried to enumerate
the irreducible graphs using just this obsevation, we would run into troubles,
since it is possible to repeat say operation (f) indefinitely, and without a way to check whether
the considered graph has an irreducible supergraph, we would not be able to ensure termination of the process.
To solve this issue, we exploit the fact that irreducible graphs are linked.

Let $H$ be a graph with a $2$-cell drawing on torus.
Let $g$ be an unlinked 4-face of $H$ with facial walk $v_1v_2v_3v_4$, where the pair $v_1$ and $v_3$ is unlinked.
Let $H'$ be a supergraph of $H$ containing a path $P$ of length three
from $v_1$ to $v_3$ such that the $5$-cycle obtained as the concatenation of $P$ with $v_1v_2v_3$ is non-contractible.
If $H'=H+P$, then we say that $H'$ is a \emph{linkage} of $H$.
Note that a linkage of a graph $H$ is obtained from $H$ by adding at most one $(\le\!3)$-chord and
at most three $1$-chords to its faces.  

\begin{lemma}\label{IrBuild}
Let $G$ be an irreducible graph drawn in torus.  Let $H$ be a connected subgraph of $G$
containing the base graph as a subgraph.  For any $4$-face $g$ and an unlinked pair $v_1$ and $v_3$
of its opposite vertices, there exists a link $P$ from $v_1$ to $v_3$ such that
$H+P\subseteq G$.
\end{lemma}
\begin{proof}
By Corollary~\ref{Separations}, $g$ is also a face of $G$.  The graph $G$ is linked by Lemma~\ref{IrLinked}, and thus $G$ contains
a link $P\not\subseteq H$ joining $v_1$ with $v_3$.  Then $H+P\subseteq G$ is a linkage of $H$.
\end{proof}

Let us now focus on the algorithmic side of enumeration.  First, let us introduce a technicality.
Let $\text{Unlink}(H)$ be the set of quadruples $(v_1,\ldots,v_4)$ such that $v_1v_2v_3v_4$ is a $4$-cycle
in $H$ bouding a face and the pair $v_1$ and $v_3$ is not linked.
Let $N$ be a set of faces of $H$.
A sequence $W=(z_1,\alpha_1,\ell_1,z_2,\alpha_2,\ell_2,\ldots, z_k)$ is an \emph{$(H,N,v_1v_2v_3v_4)$-link approximation} if
$z_1=v_1$, $z_k=v_3$, $\{v_2,v_4\}\cap \{z_1,\ldots, z_k\}=\emptyset$,
$\ell_1$, \ldots, $\ell_{k-1}$ are positive integers summing to $3$,
and for $i=1,\ldots, k-1$, either $\alpha_i$ is an edge of $H$ joining $z_i$ with $z_{i+1}$ and $\ell_i=1$,
or $\alpha_i$ is a $(\ge\!6)$-face of $H$ incident with $z_i$ and $z_{i+1}$ and not belonging to $N$
and the distance between $z_i$ and $z_{i+1}$ in the facial walk of $\alpha_i$ is at least $4-\ell_i$.
Let $\text{link}_{H,N}(v_1,v_2,v_3,v_4)$ be the number of $(H,N,v_1v_2v_3v_4)$-link approximations.
Clearly, we have the following.
\begin{observation}\label{obs-ubound}
Let $H$ be a triangle-free graph with a $2$-cell drawing on torus, let $N$ be the set of its faces,
and let $v_1v_2v_3v_4$ be a cycle in $H$ bounding a $4$-face.  Then there exist at most $\text{link}_{H,N}(v_1v_2v_3v_4)$
triangle-free linkages $H'$ of $H$ such that $H'=H+P$ for a link $P$ joining $v_1$ with $v_3$,
all faces of $N$ are also faces of $H'$, and every contractible $(\le\!5)$-cycle in $H'$ bounds a face.
\end{observation}
The function $\text{link}$ is easy to compute by brute-force enumeration, and we use it
in the algorithm as a heuristic estimate for the number of links satisfying the properties listed in Observation~\ref{obs-ubound}
(which would be somewhat slower to determine exactly).

We are now ready to formulate the algorithm to enumerate toroidal irreducible graphs; see Algorithm~\ref{alg-enum}.
The algorithm searches through the space of all irreducible graphs and their subgraphs based on Lemmas~\ref{IrLinked}, \ref{IrExpChar}, and \ref{IrBuild},
adopting a number of pruning mechanism to make it efficient.  In particular, we keep track of the set of faces $N$
into which we already tried to expand, and make sure we do not expand into them again.  We also prune branches
for isomorphic graphs (in their drawing on torus), as well as for graphs containing non-facial contractible $(\le\!5)$-cycles
(which do not have irreducible supergraphs by Corollary~\ref{Separations}).  Finally, we try to link faces with
$\text{link}_{G,N}$ as small as possible and to expand into as short faces as possible to limit the branching factor of the recursion.

\begin{Algorithm}
\begin{algorithmic}
\Function {Enumeration}{Graph $G$, faceset $N$}
\State Mark $G$ as encountered.
\If {$G$ is linked} 
\LIf {$G$ is 4-critical} output $G$ and return \EndLIf
\State $F \colonequals \{f'$ : $ f'$ is face of $G$, $f' \notin N$, $|f'| \geq 6\}$
\LIf {$F = \emptyset$} return \EndLIf
\State $f \colonequals \text{argmin} \{|f'| : f' \in F\}$
\State $\Psi \colonequals$ all triangle-free expansions of $f$ in $G$ with no non-facial contractible $(\le\!5)$-cycles
\LForAll {$H \in \Psi$ such that no isomorphic graph was encountered} 
\Call {Enumeration}{$H$,$N$} 
\EndLFor
\State \Call {Enumeration}{$G$,$N \cup \{f\}$}
\EndIf
\If {$G$ is unlinked}
\State $(v_1,v_2,v_3,v_4) \colonequals \text{argmin} \{\text{link}_{G,N}(K) : K\in \text{Unlink}(G)\}$
\State $n \colonequals \text{link}_{G,N}(v_1,v_2,v_3,v_4)$
\LIf {$n = 0$} return \EndLIf
\State $\Psi \colonequals$ all triangle-free linkages of $v_1$ and $v_3$ in $G$ with no non-facial contractible $(\le\!5)$-cycles
\LForAll {$H \in \Psi$ such that no isomorphic graph was encountered} 
\Call {Enumeration}{$H$,$N$}
\EndLFor
\EndIf
\EndFunction
\end{algorithmic}
\caption{An algorithm to enumerate irreducible graphs on torus.}\label{alg-enum}
\end{Algorithm}

\begin{lemma}\label{lemma-outputall}
Let $B$ be the base graph.  If the call of $\text{\textsc{Enumeration}}(B,\emptyset)$ from Algorithm~\ref{alg-enum}
finishes in finite time,
then it outputs exactly all irreducible graphs drawn in torus.
\end{lemma}
\begin{proof}
The graphs output by the algorithm are 4-critical, triangle-free, and linked, and thus they are irreducible.
Hence, it suffices to argue that all irreducible toroidal graphs are output.

Let $G_0$ be an irreducible graph drawn in torus.  Let $G\subseteq G_0$ be a connected subgraph such that $B\subseteq G$
and let $N$ be a set of faces of $G_0$ such that each face in $N$ is also a face of $G$.  In this situation, we say that
$(G,N)$ \emph{aims towards $G_0$}.  Suppose that the function \textsc{Enumeration} is called with arguments $(G,N)$
during the execution of $\text{\textsc{Enumeration}}(B,\emptyset)$.  The sequence $(B,\emptyset)=(G_0,N_0)$, $(G_1,N_1)$, \ldots, $(G_k,N_k)=(G,N)$
is the \emph{call stack} of $(G,N)$ if exactly the calls to $\text{\textsc{Enumeration}}(G_0,N_0)$, \ldots, $\text{\textsc{Enumeration}}(G_k,N_k)$
in order started but have not finished yet at this point.

We prove the following claim:
\begin{itemize}
\item[(GEN)] If during the execution of $\text{\textsc{Enumeration}}(B,\emptyset)$, the function \textsc{Enumeration}
is called with arguments $(G,N)$, and $G_0$ is an irreducible graph drawn in torus aimed towards by $(G,N)$, then $G_0$ is
output before this call terminates.  If a graph $H\in \Psi$ such that $(H,N)$ aims towards $G_0$ is considered during this call
and a graph isomorphic to $H$ was encountered before, then $G_0$ was output already before the call to $\text{\textsc{Enumeration}}(G,N)$.
\end{itemize}
The claim of the lemma then follows by (GEN) for $(B,\emptyset)$.

Let us now argue that (GEN) holds.  Let $G$, $N$, and $G_0$ be as in the claim.  We can assume that (GEN) holds for all
calls that end before the termination of the call to $\text{\textsc{Enumeration}}(G,N)$.  The claim is obvious when $G=G_0$,
and thus assume that $G_0$ is a proper subgraph of $G$.  Let us first consider the case that $G$ is linked and the face $f$
selected by the algorithm is also a face of $G_0$.  Then $(G,N \cup \{f\})$ aims towards $G_0$ and $\text{\textsc{Enumeration}}(G,N\cup \{f\})$
is called before the termination of the call to $\text{\textsc{Enumeration}}(G,N)$, and thus $G_0$ is output.
Furthermore, no expansion of $G$ inside $f$ is a subgraph of $G$, and thus the second part of the claim (GEN) holds vacuously. 

Hence, we can assume that either $G$ is linked and the face $f$ selected by the algorithm is not a face of $G_0$,
or $G$ is not linked.  In either case, Lemmas~\ref{IrExpChar} and \ref{IrBuild} ensure that the set $\Psi$ contains
a proper supergraph $H$ of $G$ such that $(H,N)$ aims towards $G_0$.  If a graph isomorphic to $H$ was not encountered before,
then the algorithm calls $\text{\textsc{Enumeration}}(H,N)$ before the termination of the call to $\text{\textsc{Enumeration}}(G,N)$,
and $G_0$ is output.

Hence, suppose that there has been a call to $\text{\textsc{Enumeration}}(H,N')$ for some set $N'$ of faces of $H$ that occured before
the call to $\text{\textsc{Enumeration}}(G,N)$.  In this case, we need to argue that $G_0$ was output already before the call to $\text{\textsc{Enumeration}}(G,N)$.
Since $H$ is a proper supergraph of $G$, we conclude that the call to $\text{\textsc{Enumeration}}(H,N')$
finished before the call to $\text{\textsc{Enumeration}}(G,N)$.  If $(H,N')$ aims towards $G_0$, it follows that $G_0$ was output at that point.
Hence, suppose that $N'$ contains a face of $H$ which is not a face of $G$.  Let $(H',N'')$ be the last element in the call stack
of $(H,N')$ such that $N''$ only contains faces of $G_0$.  Observe that $H'$ is linked, during the call to $\text{\textsc{Enumeration}}(H',N'')$,
the algorithm chose a face $f$ of $H'$ to expand into such that $f$ is not a face of $G_0$, and $(H', N''\cup \{f\})$ is the following
element in the call stack of $(H,N')$.  However, by Lemma~\ref{IrExpChar}, during the call to $\text{\textsc{Enumeration}}(H',N'')$,
the set $\Psi$ contained a graph $H''$ such that $(H'',N'')$ aims to $G_0$, and thus by (GEN) either $G_0$ was output
already at that point or during the call to $\text{\textsc{Enumeration}}(H'',N'')$.
\end{proof}

We implemented Algorithm~\ref{alg-enum}, the implementation can be found at \url{http://iuuk.mff.cuni.cz/~pekarej/papers/torus-irrs/}.
The resulting program finished running in finite time and produced exactly the graphs
depicted in Figure~\ref{figIr}.  By Lemma~\ref{lemma-outputall}, this implies that Theorem~\ref{IrEnum4} holds\footnote{The other co-author wrote an independent program following a slightly different enumeration scheme,
which confirms this result.}.

\section{Properties of $4$-critical toroidal triangle-free graphs}\label{SecFac}

Clearly, by repeated reductions of a $4$-critical toroidal triangle-free graph, we eventually
reach an irreducible graph.  Consequently, properties of irreducible graphs preserved by
the inverse process to reduction hold for all $4$-critical toroidal triangle-free graphs.
Consider a graph $G$ drawn in torus. The \emph{representativity} of $G$ is the
minimum number of intersections of $G$ with a non-contractible
closed curve on torus.

\begin{corollary}\label{Repr1}
Every $4$-critical triangle-free graph drawn in torus has representativity at least $2$.
\end{corollary}
\begin{proof}
Let $G$ be a $4$-critical triangle-free graph drawn in torus, and suppose for a contradiction
that $G$ has representativity at most $1$.  Choose such a graph $G$ with the smallest number
of vertices.  Let $c$ be a simple non-contractible closed curve intersecting $G$ in at most one vertex.

All the graphs depicted in Figure~\ref{figIr} have representativity at least two; hence, by Theorem~\ref{IrEnum4},
$G$ is not irreducible.  Hence, there exists a $4$-face $f$ of $G$, a triangle-free $f$-deflation $G'$ of $G$,
and a $4$-critical subgraph $H$ of $G'$.  Since $G$ has no loops or parallel edges, the face $f$ is incident with four distinct
vertices, and thus the curve $c$ does not intersect $f$.  Hence, $c$ (possibly shifted slightly in case that $c$ passes
through one of the vertices of $f$ that was identified with the opposite vertex during the $f$-deflation) intersects $G'$
in at most one vertex.  Since $H$ is a subgraph of $G$, $c$ also intersects $H$ in at most one vertex, showing that $H$
has representativity at most $1$.  This contradicts the minimality of $G$.
\end{proof}

Consequently, every triangle-free graph drawn in torus with representativity
at most $1$ is $3$-colorable.  We also obtain the following observation.

\begin{corollary}\label{IrFaceCyc}
Let $G$ be a triangle-free 4-critical graph drawn in torus. Then every face of $G$ is bounded by a cycle. 
\end{corollary}
\begin{proof}
The drawing of $G$ is $2$-cell, as otherwise $G$ would be planar and $3$-colorable by Grötzsch' theorem.
If any vertex $v$ appears twice in the facial walk of any face $f$, then consider the simple closed curve $c$
tracing (inside $f$) the part of the facial walk between the two appearances and then crossing $v$.
Since $c$ intersects $G$ only in one vertex, Corollary~\ref{Repr1} implies that $c$ is contractible.
Then $G-v$ is disconnected (with one part contained in the open disk bounded by $c$
and the other part contained in its complement).  However, this is a contradiction, since all $4$-critical graphs are
$2$-connected.
\end{proof}

To analyze the reverse process to reduction in more detail, we now relate a reduction $H$ of a
reducible graph $G$ drawn in torus to a subgraph $G_1$ of $G$ such that the faces of $H$ naturally
correspond to faces of $G_1$.
\begin{lemma}\label{lemma-unredu}
Let $G$ be a reducible graph drawn in torus, and let $H$ be its reduction.
There exists a subgraph $G_1$ of $G$ with a $2$-cell drawing
and a path $u_1vu_2$ contained in the boundary of a face $f_0$
of $G_1$, such that $f_0$ is not a face of $G$ and
\begin{itemize}
\item[\textrm{(i)}] $H$ is obtained from $G_1$ by identifying $u_1$ with $u_2$ to a new vertex $u$ within $f_0$
and suppresing the resulting $2$-face $uv$, or
\item[\textrm{(ii)}] $H$ is obtained from $G_1$ by identifying $u_1$ with $u_2$ to a new vertex $u$ within $f_0$
and deleting both resulting edges between $u$ and $v$, or
\item[\textrm{(iii)}] $v$ has degree two and it is incident with two distinct faces $f_0$ and $f_1$ in $G_1$,
$f_1$ is not a face of $G$, and $H$ is obtained from $G_1$ by contracting both edges incident with $v$.
\end{itemize}
\end{lemma}
\begin{proof}
Let $G'$ be an $f$-deflation of $G$ of form $(u_1vu_2x) \rightarrow (vux)$ such that $H$ is a subgraph of $G'$.
Since $H$ is $4$-critical and $H-u$ is a subgraph of $G$, we conclude that $u\in V(H)$.
Let $G_0$ be the subgraph of $G$ obtained from $H$ by splitting $u$ back into $u_1$ and $u_2$; we replace $u$ by $u_1$ or $u_2$
as appropriate in the edges incident with $u$ distinct from $uv$ and $ux$, while the edges $uv$ or $ux$ (if present in $H$)
are replaced by both $u_1v$ and $u_2v$, or both $u_1x$ and $u_2x$, respectively.  Let $G_1=G_0+u_1vu_2-u_2x$, and let $f_0$
be the face of $G_1$ such that $f\subset f_0$.  Note that $f_0$ is not a face of $G$, since $u_2x$ is part of the boundary of $f$.
Furthermore, since the drawing of $H$ is $2$-cell, observe that the drawing of $G_1$ is $2$-cell as well.

If $uv\in E(H)$, then $H$ is obtained from $G_1$ as described in (i).  If $uv\not\in E(H)$, but $v\in V(H)$, then
$H$ is obtained from $G_1$ as described in (ii).  If $v\not\in V(H)$, then $v$ has degree two in $G_1$ and $H$ is obtained from $G_1$
by contracting both edges incident with $v$.  Since $H$ is not a subgraph of $G$, both $u_1$ and $u_2$ have degree at least two in $G_1$.
Let $f_0$ and $f_1$ be the two faces of $G_1$ incident with $v$.  If $f_0=f_1$, then $u$ appears twice in the boundary walk of the corresponding
face of $H$, which contradicts Corollary~\ref{IrFaceCyc}.  Hence, the two faces are distinct.  If $f_1$ were a face of $G$,
then the path $u_1vu_2$ would be contained in the boundaries of faces $f_1$ and $f$ of $G$, and thus $v$ would have degree two,
which is a contradiction since $G$ is $4$-critical.  Hence, (iii) holds.
\end{proof}
Observe that there is a natural 1-to-1 correspondence between the faces of $H$ and $G_1$, except for the case (ii),
where the face of $H$ whose interior contains the (deleted) edge $uv$ corresponds to the two faces of $G_1$
incident with $u_1v$ and $u_2v$ distinct from $f_0$.

If $G$ is a graph drawn in a surface $\Sigma$ and $g$ is a $2$-cell face of its subgraph, let $G^0_g$ denote the
subgraph of $G$ drawn in the closure of $g$ and let $C^0_g$ denote the subgraph of $G$ drawn in the boundary of $g$.
In order to derive information about $G$ from the subgraph $G_1$ as in Lemma~\ref{lemma-unredu},
we need to say something about the graphs $G^0_g$ for faces $g$ of $G_1$.
By Lemma~\ref{SubgrCrit}, either $g$ is a face of $G$ or $G^0_g$ is $C^0_g$-critical.
Note that if $C^0_g$ is a cycle, then $G^0_g$ is a plane graph and $C^0_g$ bounds one of its
faces.  We want to view $G_g$ similarly also in the general case that some vertex appears more than once in $C^0_g$. 
This is possible via the following construction and lemma.  Let $\Delta$ be a closed disk and let $\theta:\Delta\to\Sigma$ be a continuous
function such that the restriction of $\theta$ to the interior of $\Delta$ is a homeomorphism with $g$.  Let $G_g=\theta^{-1}(G)$.
Note that the boundary of $\Delta$ traces a cycle $C_g$ in $G_g$ such that $\theta(C_g)=C^0_g$.

\begin{lemma}\label{lemma-facrit}
Let $G$ be a graph drawn in a surface and let $g$ be a $2$-cell face of its subgraph.  If $G$ is $4$-critical and $g$ is not a face of $G$,
then $G_g$ is $C_g$-critical.
\end{lemma}
\begin{proof}
Let $\theta$ be as in the definition of $G_g$.   Consider a proper subgraph $H$ of $G_g$ with $C_g\subseteq H$.  The graph $G^0_g$ is $C^0_g$-critical
by Lemma~\ref{SubgrCrit}, and thus there exists a $3$-coloring $\psi$ of $C^0_g$ that extends to a $3$-coloring $\varphi$ of $\theta(H)$, but not to a $3$-coloring of $G^0_g$.
Then $\theta\circ\psi$ is a $3$-coloring of $C_g$ that extends to a $3$-coloring $\theta\circ\varphi$ of $H$, but not to a $3$-coloring of $G_g$.
\end{proof}

Triangle-free plane graphs critical with respect to a facial cycle
have been intensively studied.  Let us recall some relevant results.
A graph $G$ embedded in the plane has one unbounded (outer) face;
all other faces of $G$ are \emph{internal}. For a graph $G$ embedded in the
plane let $S(G)$ denote the multiset of lengths of \emph{internal} $(\geq 5)$-faces of
$G$.  Let $\mathcal{G}_{\gamma,k}$ denote the set of all plane graphs $G$ of girth at least
$\gamma$ and with outer face bounded by a cycle $C$ of length $k$ such that $G$ is $C$-critical.
Let $\mathcal{S}_{\gamma,k}$ denote set $\{S(H) : H \in \mathcal{G}_{\gamma,k} \}$. 

Let $S_1$ and $S_2$ be multisets of integers. We say that $S_2$ is a \emph{one-step refinement} of
$S_1$ if there exists $k \in S_1$ and a set $Z \in \mathcal{S}_{4,k} \cup
\mathcal{S}_{4,k+2}$ such that $S_2 = (S_1 \backslash \{k\}) \cup Z$. We say
that $S_2$ is a \emph{refinement} of $S_1$ if it can be obtained from $S_1$ by
a (possibly empty) sequence of one-step refinements. 
Dvořák, Král, and Thomas~\cite{trfree4} proved the following lemma linking the
possible lengths of faces in critical graphs of girths $4$ and $5$.

\begin{lemma}[\cite{trfree4}]\label{AmpS4Formula} 
For every $k \geq 7$, each element of $\mathcal{S}_{4,k}$ other than $\{k-2\}$
is a refinement of an element of $\mathcal{S}_{4,k-2} \cup \mathcal{S}_{5,k}$.
In particular, if $S \in \mathcal{S}_{4,k}$ then the maximum of $S$ is at most
$k-2$, and if the maximum is equal to $k-2$, then $S = \{k-2\}$.
\end{lemma}

To make use of this claim, we need information about the face sizes in
critical graphs of girth $5$.  The relevant cases are implied by a result of
Thomassen~\cite{thom-torus}, which can be stated as follows.

\begin{theorem}[\cite{thom-torus}]\label{AmpS5}
Let $G$ be a planar graph of girth at least $5$ and with the outer face bounded by
a cycle of length at most $9$.  If $G$ is $C$-critical, then
either $|C|\ge 8$ and $G$ consists of $C$ and its chord, or
$|C|=9$ and $G$ consists of $C$ and a vertex with three neighbors in $C$.
\end{theorem}

We can now combine these results.

\begin{lemma}\label{AmpS4}
The following claims hold.
\begin{itemize}
\item[\textrm{(i)}] $\mathcal{S}_{4,4} = \mathcal{S}_{4,5} = \emptyset$ and $\mathcal{S}_{4,6} \subseteq \{\emptyset \}$
\item[\textrm{(ii)}] $\mathcal{S}_{4,7} \subseteq \{\{5\}\}$
\item[\textrm{(iii)}] $\mathcal{S}_{4,8} \subseteq \{\{6\}, \{5,5\}, \emptyset \}$
\item[\textrm{(iv)}] $\mathcal{S}_{4,9} \subseteq \{\{7\}, \{6,5\}, \{5,5,5\}, \{5\}\}$
\end{itemize}
\end{lemma}
\begin{proof}
Theorem~\ref{Outer} implies (i).  By Theorem~\ref{AmpS5}, we have
$\mathcal{S}_{5,k} = \emptyset$ for $k \leq 7$,
$\mathcal{S}_{5,8} \subseteq \{ \{5,5\} \}$, and
$\mathcal{S}_{5,9} \subseteq \{ \{6,5\}, \{5,5,5\}\}$.
The parts (ii)--(iv) then follow by Lemma \ref{AmpS4Formula}. 
\end{proof}

Using Lemma~\ref{lemma-unredu}, we can relate face sizes in a reduction
to the face sizes in the original graph.  Let us introduce two operations
on multisets of face lengths.  The first of them (splitting) corresponds
to adding a chord to one of the faces, the second of them (amplification)
corresponds to pasting critical triangle-free graphs into some of the faces.

Let $I$ be a multiset of integers.
A multiset $A$ is a \emph{splitting}
of $I$ if $A=I$,
or $A$ is obtained from $I$ by removing one appearance of element $6$,
or $A$ is obtained from $I$ by replacing an element $i\ge 7$ by $i-2$,
or $A$ is obtained from $I$ by replacing an element $i\ge 8$ by two elements $i_1$ and $i_2$
such that $i_1,i_2\ge 5$ and $i_1+i_2=i+2$.
A multiset $A$ is an \emph{amplification} of
$I$ if $A$ is obtained from $I$ by replacing some elements $i$ of $I$ by
the elements of some multiset from $\mathcal{S}_{4,i}\cup \mathcal{S}_{4,i+2}$.

\begin{lemma}\label{AmpFaceSet}
Let $G$ be a reducible graph drawn in torus and let $H$ be its reduction.
Then $S(G)$ is an amplification of a splitting of $S(H)$.
\end{lemma}
\begin{proof}
Let $G_1$ be a subgraph of $G$, $u_1vu_2$ a path, and $f_0$ (and possibly $f_1$) faces
as in Lemma~\ref{lemma-unredu}.
Let $H_0=H+uv$ if the conclusion (ii) of Lemma~\ref{lemma-unredu} holds and $H_0=H$ otherwise;
clearly, $S(H_0)$ is a splitting of $S(H)$.

Let $Z_0$ be the set of $(\ge\!5)$-faces $g$ of $G_1$ such that $g$ is also a face of $G$.
Let $B_0=\{f_0\}$ if the conclusion (i) or (ii) of Lemma~\ref{lemma-unredu} holds, and let
$B_0=\{f_0,f_1\}$ if the conclusion (iii) holds.  Let $A_0$ be the set of faces $g$ of $G_1$ such that $g\not\in B_0$
and $g$ is not a face of $G$.  Let us define multisets $Z=\{|g|:g\in Z_0\}$, $A=\{|g|:g\in A_0\}$ and $B=\{|g|-2:g\in B_0, |g|\ge 7\}$.

Note that all faces in $A_0\cup Z_0$ are also faces of $H_0$ and those in $A_0$ have length at least $6$ by Corollary~\ref{Separations}.
Furthermore, the faces of $H_0$ corresponding to $f_0$ (and possibly $f_1$) have lengths $|f_0|-2$ (and $|f_1|-2$).
Hence, $S(H_0)=Z\cup A\cup B$.

Each face of $G$ of length at least $5$ either belongs to $Z_0$ or it is (a homeomorphic image of) a face of
the graph $G_g$ for some face $g\in A_0\cup B_0$.  By Lemma~\ref{lemma-facrit}, $G_g$ is $C_g$-critical,
and thus $S(G_g)\in \mathcal{S}_{4,|g|}$.  Furthermore, if $g\in B_0$ and $|g|\le 6$, then $S(G_g)=\emptyset$ by Lemma~\ref{AmpS4}(i).
We conclude that the length of each face of $G$ belongs to $Z$ or to $\mathcal{S}_{4,i}$ for some $i\in A$ or to $\mathcal{S}_{4,i+2}$ for some $i\in B$,
and consequently $S(G)$ is an amplification of $Z\cup A\cup B=S(H_0)$.
\end{proof}

We are now ready to prove our main result, the description of lengths of faces
of all triangle-free 4-critical graphs drawn in torus, up to the number of
their 4-faces. 

\begin{proof}[Proof of Theorem~\ref{thm-main}]
Note that $G$ has representativity at least $2$ by Corollary~\ref{Repr1}.
We prove that $S(G)$ is as claimed and that $c(G)\ge 7$ by induction on the number of vertices of $G$.
Both claims hold for the irreducible graphs depicted in Figure~\ref{figIr}, and thus
suppose that $G$ is reducible and the conclusions of Theorem~\ref{thm-main} hold for all graphs with fewer vertices than $G$.
Let $H$ be a reduction of $G$.  By Lemma~\ref{4FaceDec}, we have $c(G)>c(H)\ge 7$, as required.

Furthermore, by induction hypothesis, $S(H)$ is one of the multisets listed in the statement of Theorem~\ref{thm-main}.
By Lemma~\ref{AmpFaceSet}, $S(G)$ is an amplification of a splitting of one of these sets.
The possible splittings of these multisets are:
\begin{itemize}
\item $\{7,5\}\to \{7,5\}$ or $\{5,5\}$,
\item $\{6,5,5\}\to \{6,5,5\}$ or $\{5,5\}$,
\item $\{5,5,5,5\}\to \{5,5,5,5\}$,
\item $\{5,5\}\to \{5,5\}$, and
\item $\emptyset \to \emptyset$
\end{itemize}
Using Lemma~\ref{AmpS4}, the amplifications of these multisets are:
\begin{itemize}
\item $\{7,5\}\to \{7,5\}$ or $\{6,5,5\}$ or $\{5,5,5,5\}$ or $\{5,5\}$,
\item $\{6,5,5\}\to \{6,5,5\}$ or $\{5,5,5,5\}$ or $\{5,5\}$,
\item $\{5,5,5,5\}\to \{5,5,5,5\}$,
\item $\{5,5\}\to \{5,5\}$, and
\item $\emptyset \to \emptyset$.
\end{itemize}
Hence, $S(G)$ is one of the multisets listed in the statement of Theorem~\ref{thm-main}.
Furthermore, if $S(G)=0$, then $G$ is a quadrangulation of the torus and $G=I_4$ by~\cite{thomas2008coloring}.
\end{proof}

\end{document}